\documentclass[a4paper,11pt]{article}
\usepackage{euscript,graphicx,epstopdf,amscd,amsgen,amsfonts,amssymb,latexsym,amsmath,amsthm,graphicx,mathrsfs,times,color,overpic,bbm}
\usepackage{enumitem}

\usepackage[hidelinks]{hyperref}

\usepackage{xcolor}

\usepackage{comment}

\definecolor{forestgreen}{rgb}{0.0, 0.27, 0.13}

\newtheorem{theorem}{Theorem}

\newtheorem{lemma}[theorem]{Lemma}
\newtheorem{corollary}[theorem]{Corollary}
\newtheorem{proposition}[theorem]{Proposition}

\theoremstyle{definition}
\newtheorem{example}[theorem]{Example}
\newtheorem{remark}[theorem]{Remark}

\def\bN{\mathbb{N}}

\def\bS{\mathbb{S}}
\def\bQ{\mathbb{Q}}

\def\cA{\mathcal{A}}

\def\dd{\mathrm{d}}

\DeclareMathOperator{\Id}{Id}

\DeclareMathOperator{\diam}{diam}
\DeclareMathOperator{\Int}{Int}

\DeclareMathOperator{\diff}{Diff}
\DeclareMathOperator{\supp}{supp}

\DeclareMathOperator{\GL}{GL}

\newcommand{\N}{\mathbb{N}}
\newcommand{\Z}{\mathbb{Z}}
\newcommand{\R}{\mathbb{R}}

\newcommand{\LimS}{\displaystyle\limsup}

\newcommand\norm[1]{\left\lVert#1\right\rVert}

\def\supp{\operatorname{supp}}

\def\dim{\operatorname{dim}}

\def\diff{\operatorname{Diff}}
\def\hom{\operatorname{Hom}}

\def\min{\operatorname{min}}

\def\diam{\operatorname{diam}}

\DeclareMathSymbol{\emptyset}{\mathord}{AMSb}{"3F}

\providecommand{\keywords}[1]{%
  \vspace{0.5em}
  \noindent\textbf{Keywords:} #1
}

\providecommand{\msc}[1]{%
  \vspace{0.5em}
  \noindent\textbf{2020 Mathematics Subject Classification:} #1
}

\usepackage{authblk}
\author[1]{Jamerson Bezerra}
\author[2]{Graccyela Salcedo}
\affil[1]{
    Universidade Federal do Ceará, Av. da Universidade, 2853 - Benfica, 60020-181 Fortaleza - CE, Brazil\\
    
    {\small (email: \texttt{jdbezerra@mat.ufc.br})}
}
\affil[2]{
    Centre de Physique Théorique, CNRS, Institut Polytechnique de Paris, Palaiseau, France\\
    
    {\small (email: \texttt{graccyela.salcedo@polytechnique.edu}) }
}

\usepackage[margin=1.2in]{geometry}

\begin{document}

\title{A version of Oseledets for proximal random dynamical systems on the circle}
\date{\vspace{-5ex}}

\maketitle

\begin{abstract}
    We study proximal random dynamical systems of homeomorphisms of the circle without a common fixed point. We prove the existence of two random points that govern the behavior of the forward and backward orbits of the system. Assuming the differentiability of the maps, we characterize these random points in terms of the extremal Lyapunov exponents of the random dynamical system. As an application, we prove the exactness of the stationary measure in this setting.
\end{abstract}


\msc{
    37H15, 37E10, 37Hxx, 60G50, 37B05.
}

\section{Introduction}
We study the hyperbolic behavior of orbits in random dynamical systems (RDSs) on $\bS^1$. Over the past sixty years, the theory of RDSs has evolved significantly, driven by the works of Furstenberg \cite{Fur:73}, Kifer \cite{Kif:2012}, Arnold \cite{Arn1995}, and many others. These contributions have established a strong interplay between deterministic and probabilistic dynamical systems.

A central class of RDSs arises from linear cocycles. These are systems induced by
the action of $\GL_2(\R)$ matrices on $\bS^1$. The dynamical and statistical properties of these systems are often connected with  behavior of (extremal) Lyapunov exponents, which measure the asymptotic exponential growth rate of contraction or expansion in the circle action, as formalized by Oseledets’ theorem \cite{Os1968}. See \cite{Led:84} for a detailed discussion on Oseledets' theorem and its applications. For a modern presentation see \cite{Via:14}.

In the linear case, Oseledets’ theorem also asserts that if there is a gap between the two Lyapunov exponents of the system, then there exists a pair of random directions in $\bS^1$, called \emph{Oseledets directions}, denoted $e^s$ and $e^u$, that characterize the dynamics of the action. Specifically, $e^u$ and $e^s$ attract the forward and backward orbits of the cocycle, respectively (see \cite{Via:14}). Our primary goal is to extend Oseledets’ result to the setting of nonlinear RDSs generated by homeomorphisms on $\bS^1$. We provide sufficient conditions ensuring the existence of random directions that govern the system’s asymptotic behavior.

Unlike RDSs induced by linear maps, nonlinear effects introduced by general homeomorphisms can cause deviations at small scales, adding complexity to the study of random system orbits. These effects give rise to new dynamical behavior, which has recently drawn considerable attention (see \cite{Mal:17}, \cite{GelSal:23}).

We build on the work of \cite{Mal:17} to study the topological properties of proximal RDSs. Proximality was first introduced by Furstenberg \cite{Fur:73} in the context of defining the boundary of Lie groups. In \cite{GuiRau:86}, Guivarc’h and Raugi established the typicality of the proximality condition among RDSs induced by linear maps.

It is well established that under the proximality condition, RDSs exhibit a broad spectrum of statistical properties. In particular, in \cite{GelSal:24} showed that if proximality is combined with a local contraction property, then the system satisfies a strong law of large numbers, a central limit theorem, and a law of the iterated logarithm (see \cite[Theorem 1.6]{GelSal:24}). Besides their theoretical relevance, these results provide a solid framework to interpret the asymptotic behavior of random orbits and to support the validation of numerical simulations against the expected statistical laws.

We prove that a proximal RDS without a common fixed point admits a pair of random Oseledets directions. Under an additional regularity assumption, we characterize these directions in terms of the extremal Lyapunov exponents. Finally, we establish the exact dimensionality of the associated stationary measures, extending the linear case studied by Hochman and Solomyak \cite{HocSol:17}. Below, we present the precise statements of our results.

\subsection{Statement of the results}
Denote by $\hom(\bS^1)$ the space of the circle homeomorphisms by $\hom(\bS^1)$ and let $d$ be the usual metric on $\bS^1 := \R/\Z$. Each probability measure $\nu$ on $\hom(\bS^1)$ induces a \emph{random dynamical system} (RDS) on $\bS^1$ produced by random compositions of the form $f_n\circ\cdots\circ f_1$, $n\in \N$, with $f_i$ chosen accordingly with the distribution $\nu$.

Denote by $X_{\nu}$ the topological support of $\nu$. Let $\Gamma_\nu\subset \hom(\bS^1)$ be the semigroup generated by $X_{\nu}$. We say that the measure $\nu$, or the $\Gamma_{\nu}$ action, or equivalently the RDS induced by $\nu$, is \emph{proximal} if for all $x,y\in \bS^1$ there exists a sequence $(g_n)_{n\in\bN}$ of elements in $\Gamma_{\nu}$ such that
\begin{align*}
    \lim_{n\to\infty}d(g_n(x), g_n(y)) = 0.
\end{align*}

On top of the proximality assumption, we assume frequently in this work that the semigroup $\Gamma_{\nu}$ generated by the support $\nu$ does not fix a point in $\bS^1$. This means that there exists no $x\in \bS^1$ such that $f(x) = x$ for every $f\in \Gamma_{\nu}$. This is the natural condition on the proximal system to ensure dynamical complexity and play a role analogous to the gap condition in the Oseledets' theorem (see Remark \ref{rk:noInvariantMeasure}).

For the following, denote by $\sigma: X_\nu^\bN\to  X_\nu^\bN$ the \emph{left-shift map}, that is, $\sigma\omega = (f_n)_{n\geq 2}$. Our first main result generalizes Oseledets' theorem to general RDSs of circle homeomorphisms.
\begin{theorem}\label{thm:1}
\label{teo:01}
    Let $\nu$ be a probability measure on $\hom(\bS^1)$.
    Assume that $\Gamma_{\nu}$ is proximal and does not fix any point in $\bS^1$. Then there exist measurable functions $\pi,\theta: X_\nu^\bN\to\bS^1$ such that
    \begin{enumerate}
         \item for all $x\in\bS^1$ and  $\nu^\bN$-almost every sequence $\omega=(f_n)_{n\in\bN}\in X_{\nu}^{\N}$,
         $$
         \lim_{n\to\infty}f_1\circ\cdots\circ f_n(x)=\pi(\omega)
         \quad\mbox{and}\quad
         \lim_{n\to\infty}f_1^{-1}\circ\cdots\circ f_n^{-1}(x)=\theta(\omega);
         $$
         
         \item  for $\nu^\bN$-almost every sequence $\omega=(f_n)_{n\in\bN}\in X_{\nu}^{\N}$ and for each pair of closed sets $A\subset\bS^1\backslash\{\theta(\omega)\}$, $B\subset\bS^1\backslash\{\pi(\omega)\}$,
         $$
         \lim_{n\to\infty}\diam f_n\circ\cdots\circ f_1(A)=
         \lim_{n\to\infty}\diam f_n^{-1}\circ\cdots\circ f_1^{-1}(B)=0;
         $$
         \item for all $\nu^{\N}$-almost every sequence $\omega=(f_n)_{n\in\bN}\in X_{\nu}^{\N}$
        \begin{align*}
            f_1(\pi(\sigma\omega)) = \pi(\omega)
            \quad
            \text{and}
            \quad
            f_1(\theta(\omega)) = \theta(\sigma\omega).
        \end{align*}
    \end{enumerate}
\end{theorem}
Theorem \ref{thm:1} generalizes Oseledets' theorem \cite{Os1968} from RDSs induced by linear cocycles to general RDSs of homeomorphisms. The random directions $\pi$ and $\theta$ in Theorem \ref{thm:1} serve as analogs of the stable and unstable directions in Oseledets' classical result. However, we note that our setting only assumes that the maps in the support of the reference measure $\nu$ are circle homeomorphisms. Additionally, we do not require these maps to preserve or reverse orientation.

In contrast with earlier work by Kleptsyn and Nalskii \cite{KleNal:2004}, our assumptions are significantly weaker. They require minimality for both the original system and its inverse (the one induced by the measure defined in \eqref{def:eta-}), and they further assume the existence of a ``north pole/south pole" type map—namely, a homeomorphism with exactly two fixed points, one attracting and one repelling. In contrast, our assumptions are weaker: forward minimality alone ensures the absence of a common fixed point for all maps, while proximality follows from the existence of a map of north pole/south pole type. It is also worth noting that forward minimality does not imply minimality of the inverse system; see, for instance, \cite[Example 3.3.2]{BarGhaMalSar:2017}.

A probability measure $\eta$ on $\bS^1$ satisfying
\begin{align}\label{eq:def-nu-stat}
    \eta=\int_{X_\nu} f_\ast\eta\,\dd\nu(f).
\end{align}
is called \emph{$\nu$-stationary measure}.
The following corollary of Theorem \ref{thm:1} characterizes the distribution of the random ``stable" and ``unstable" directions as stationary measures.

\begin{corollary}\label{cor01:teo01}
     Let $\nu$ be a probability measure on $\hom(\bS^1)$. Assume that $\Gamma{\nu}$ is proximal and does not fix any point in $\bS^1$. Consider the measurable functions $\pi,\theta: X_\nu^\bN\to\bS^1$ as in Theorem \ref{teo:01}. Then the probability measures $\eta$ and $\eta^-$ on $\bS^1$ given by
     \[
     \eta=\int_{X_\nu^{\bN}}\delta_{\pi(\omega)}\,\dd\nu^{\bN}(\omega)\quad\mbox{and}\quad\eta^-=\int_{X_\nu^{\bN}}\delta_{\theta(\omega)}\,\dd\nu^{\bN}(\omega),
     \]
     are the unique probability measures on $\bS^1$ such that
     \begin{equation}\label{eq:cor01:teo01}
          \eta=\int_{X_\nu} f_\ast\eta\,\dd\nu(f)
          \quad
          \mbox{and}
          \quad
          \eta^-=\int_{X_\nu} (f^{-1})_\ast(\eta^-)\,\dd\nu(f).    
     \end{equation}
\end{corollary}
To further characterize the random directions $\pi$ and
$\theta$, let us now assume additional regularity in the system. Let $\diff^{1+\tau}(\bS^1)$ be the group of circle diffeomorphisms whose derivatives are Hölder continuous with exponent $\tau>0$. Let $\nu$ be a probability measure on $\diff^{1+\tau}(\bS^1)$ satisfying the following 
\emph{integrability conditions}
\begin{equation}\label{int-cond}\tag{I}
    \int \log \sup_{x\in\bS^1}\vert f'(x)\vert\,\dd\nu(f)<\infty\quad\mbox{and}\quad\int \log \inf_{x\in\bS^1}\vert f'(x)\vert\,\dd\nu(f)>-\infty.
\end{equation}
By Kingman's sub-additive ergodic theorem, we see that for $\nu^{\N}$-a.e $\omega = (f_n)_{n\in\N}\in X_{\nu}^{\N}$ the limits
\begin{align*}
    \Lambda(\nu)
    &:= \lim_{n\to\infty} \frac{1}{n}\log\sup_{z\in \bS^1}|(f_n\circ\ldots\circ f_1)'(z)|
    \quad
    \text{and}\\
    \lambda(\nu)
    &:= \lim_{n\to\infty}\frac{1}{n}\log\inf_{z\in \bS^1}|(f_n\circ\ldots\circ f_1)'(z)|.
\end{align*}
exists and are independent of $\omega$. We call the numbers $\Lambda(\nu)$ and $\lambda(\nu)$ the \emph{extremal Lyapunov exponents} of $\nu$.

By considering the associated derivative cocycle, we obtain the following generalization of Oseledets' theorem \cite{Os1968} (see also \cite{Led:84}), stated in terms of the extremal Lyapunov exponents.
\begin{theorem}\label{teo:02}
Let $\nu$ be a probability measure on $\diff^{1+\tau}(\bS^1)$ satisfying the integrability condition \eqref{int-cond}. Assume that $\Gamma_{\nu}$ is proximal and does not fix any point in $\bS^1$. Consider $\pi,\theta:X_\nu^\bN\to\bS^1$ as in Theorem \ref{teo:01}. Then,
    for $\nu^{\N}$-almost every $\omega=(f_n)_{n\in\bN}\in X_{\nu}^{\N}$,
    \[
        \lim_{n\to\infty}\frac{1}{n}\log|(f_n^{-1}\circ\cdots\circ f_1^{-1})'(x)|= \begin{cases}
            -\Lambda(\nu),&x\neq\pi(\omega)\\
            -\lambda(\nu),&x=\pi(\omega)
        \end{cases},
    \]
    and 
    \[
    \lim_{n\to\infty}\frac{1}{n}\log|(f_n\circ\cdots\circ f_1)'(x)|=\begin{cases}
        \lambda(\nu),& x\neq\theta(\omega) \\
        \Lambda(\nu),&x=\theta(\omega)
    \end{cases}.
    \]
\end{theorem}
{\color{violet}
}

As a consequence of Theorem \ref{teo:02}, we obtain a characterization of the extremal Lyapunov exponents in terms of the stationary measures of the system.
\begin{corollary}
\label{cor:teo:02}
    Let $\nu$ be a probability measure on $\diff^{1+\tau}(\bS^1)$ satisfying the integrability condition in \eqref{int-cond}.
     Assume that $\Gamma_{\nu}$ is proximal and does not fix any point in $\bS^1$. Then,
     \begin{align*}
         \Lambda(\nu) =-\int \log\vert f'(x)\vert\, \dd(\eta^-\otimes\nu)(x,f)
         \quad
         \text{and}
         \quad
         \lambda(\nu) =  \int \log\vert f'(x)\vert\, \dd(\eta\otimes\nu)(x,f),
     \end{align*}
     where $\eta$ and $\eta^-$ are as in Corollary \ref{cor01:teo01}.
\end{corollary}

Recently, Barrientos and Malicet, in \cite{BaMa2024}, established a version of Kingman’s subadditive ergodic theorem for Markov operators, using it among other things to study statistical properties of random dynamical systems comprised of Lipschitz transformations under the mostly contracting assumption—that all Lyapunov exponents associated with any stationary measure are negative. We note that the assumption that we use repeatedly fits naturally into their framework; consequently, most of their results apply to proximal systems without fix points of Lipschitz transformations.

In \cite{BaMa2024}, they also introduced a notion of maximal Lyapunov exponent for mostly contracting systems and investigated its properties. Here, we emphasize that our concept of an extremal Lyapunov exponent differs from theirs: Theirs quantifies the maximal contraction rate in mostly contracting systems and equals the largest Lyapunov exponent among stationary measures. In contrast, our extremal Lyapunov exponent measures the system’s expansion rate and is characterized as the negative of the Lyapunov exponent for the stationary measure of the inverse system.    

Let us now study the exactness of the dimension of the stationary measures.  Let $\zeta$ be a probability measure on $\bS^1$. For every $x\in \bS^1$, define the \emph{lower local dimension} $\underline{\dim}$ and \emph{upper local dimension} $\overline{\dim}$ of $\zeta$ at $x$ respectively by
\begin{align*}
    \underline{\dim}(\zeta;x) = \liminf_{r\to 0^+}\frac{
        \log(\zeta(B(x,\, r)))
    }{\log r}
    \quad
    \text{and}
    \quad
    \overline{\dim}(\zeta;x) = \limsup_{r\to 0^+}\frac{
        \log(\zeta(B(x,\, r)))
    }{\log r}.
\end{align*}
A measure $\zeta$ on $\bS^1$ is \emph{exact-dimensional}, if there exists $\alpha\geq 0$ such that for $\zeta$-almost every $x\in \bS^1$ we have that $\underline{\dim}(\zeta;x) = \overline{\dim}(\zeta;x) = \alpha$. In this case, we say that $\alpha$ is the \emph{exact-dimension} of $\zeta$ and write $ \dim(\zeta)= \alpha $. If $\zeta$ is exact-dimensional of dimension $\alpha$, then $\alpha$ is the smallest Hausdorff dimension of sets of positive $\zeta$-measure (see \cite{You:82}).

For a $\nu$-stationary measure $\eta$, we define the \emph{Furstenberg entropy} of the probability measure $\eta$ by
\begin{align*}
    h_F(\eta,\nu) :=\int\, \log\frac{
        \dd f_\ast \eta
    }{\dd\eta}(f(x))\, \dd(\eta\otimes\nu)(x,f).
\end{align*}
The entropy $h_F(\eta)$ quantifies, in a certain sense, the lack of invariance of the measure $\eta$ with respect to the maps in $\Gamma_{\nu}$. For instance, note that $h_F(\eta) = 0$ is equivalent to $f_\ast \eta = \eta$ for every homeomorphism $f\in \Gamma_{\nu}$.

\begin{theorem}\label{teo03}
Let $\nu$ be a probability measure on $\diff^{1+\tau}(\bS^1)$ satisfying \eqref{int-cond}. Assume that $\Gamma_{\nu}$ is proximal and does not fix any point in $\bS^1$. Then the $\nu$-stationary probability measure $\eta$ is exact-dimensional and 
\begin{align}
\label{eq:070325.1}
    \dim(\eta)=-\frac{h_F(\eta,\nu)}{\lambda(\nu)}.
\end{align}
\end{theorem}
Applying the Invariance Principle (see \cite{Mal:17}), we can conclude in this setting that $\lambda(\nu) \leq - h_F(\eta,\nu)<0$. In particular, by the formula \eqref{eq:070325.1} we have that $\dim(\eta) \in (0,1]$.

In \cite{HeJiXu:2023}, the authors also investigate the exact dimensionality of the stationary measure, assuming the stronger hypothesis of orientation-preserving diffeomorphisms of the circle. Both their approach and ours primarily follow the strategy developed in \cite{HocSol:17}. In a sense, our assumption of proximality allows us to circumvent the lack of orientation preservation. To achieve this, we rely heavily on the approximation results for extremal Lyapunov exponents and their associated random directions established above.

The work is structured as follows. In Section \ref{sec:01}, we exhibit some properties of proximal systems and present the proof of Theorem \ref{teo:01} and Corollary \ref{cor01:teo01}. We discuss the differentiable case in Section \ref{sec:diffCase}. Exactness of the dimension of the stationary measure is discussed in Section \ref{sec:exactDim}.

\section{Properties of proximal systems}\label{sec:01}
In this section we explore topological and measure theoretical properties of proximal systems. They will provide the tools  Theorem \ref{teo:01} and Corollary \ref{cor01:teo01}. We start by fixing some notation.

For $x,y\in \bS^1$, $x\neq y$, we write $[x,y]$ and $[y,x]$, for the two arcs connecting $x,y\in \bS^1$. Let $\nu$ be a probability measure on $\hom(\bS^1)$. As before let $X_{\nu}$ be the support of $\nu$ and write $\Gamma_{\nu}$ for the sub-semigroup of $\hom(\bS^1)$ generated by $X_{\nu}$. We say that a measure $\eta$ on $\bS^1$ is \emph{$\Gamma_{\nu}$-invariant} if $f_*\eta = \eta$, for $\nu$-almost every $f\in X_{\nu}$.

Denote by $\sigma\colon X_\nu^\bN\to  X_\nu^\bN$ the left-shift map defined on the product space $X_{\nu}^{\N}$ and consider the $\sigma$-invariant measure $\nu^{\N}$ on $X_{\nu}^{\N}$. Given $\omega = (f_n)_{n\in\bN}\in  X_\nu^\bN$ we fix the following notation
\begin{align}
\label{eq:compositionOrderNotation}
    f^n_{\omega} := f_n\circ \cdots \circ f_1,
    \quad
    \overline{f^n_{\omega}} := f_1\circ \cdots \circ f_n, \quad f^0_{\omega}=\overline{f^0_{\omega}}=\Id_{\bS^1} .
\end{align}

For proximal systems, the presence of an invariant probability measure on the circle which is invariant for all the transformations simultaneously forces the existence of a point which is fixed by the semigroup action. This fact is a consequence of the structural result \cite[Corollary]{Mal:17}, but we add an elementary proof here.
\begin{proposition}
\label{prop:invariantMeasureFixedPoint}
     Assume that $\Gamma_{\nu}$ is proximal. If $\eta$ is a $\Gamma_{\nu}$-invariant probability measure on $\bS^1$, then $\eta = \delta_{z_0}$, for some $z_0\in \bS^1$ with $g(z_0) = z_0$, for $\nu$-almost every $g\in X_{\nu}$.
\end{proposition}
\begin{proof}
    Consider $x,y\in \bS^1$ such that both arcs connecting $x$ and $y$ have positive $\eta$-measure, that is, $\eta([x,y])>0$ and $\eta([y,x])>0$. By proximality, we can find a sequence $(g_n)_{n\in\N}$ in $\Gamma_{\nu}$ such that
    \begin{align*}
        \lim_{n\to\infty}d(g_n(x),g_n(y)) = 0.
    \end{align*}
    Thus, up to considering a subsequence, we see that $(g_n)_n$ contracts one of the arcs connecting $x$ and $y$. So, we may assume without loss of generality that 
    \begin{align}
    \label{eq:101123.1}
        \lim_{n\to\infty}\mbox{diam}(g_n[x,y])=0
        \quad
        \text{and}
        \quad
        \lim_{n\to\infty}d(g_n(x),z)=\lim_{n\to\infty}d(g_n(y),z)=0,
    \end{align}
    for some $z\in \bS^1$. Let us prove that $z$ must be an atom of $\eta$, consider $\varepsilon>0$, and notice that for $n\geq 1$ sufficiently large, $g_n[x,y] \subset B(z,\varepsilon)$. Using that $\eta$ is preserved by $g_n$ for every $n\geq 1$, we have 
    \begin{align*}
        \eta(B(z,\varepsilon)) \geq \eta(g_n[x,y]) = \eta([x,y]).
    \end{align*}
    This ensure a positive lower bound for the $\eta$-measure of the ball $B(z,\varepsilon)$ independently of $\varepsilon$ implying that $\eta$ has atoms.

    We claim that every atom of $\eta$ is fixed by every element in $\Gamma_{\nu}$.  Indeed, consider a point $z\in \bS^1$, with $\eta(\{z\})>0$ and define the set $A_z$ of points $x\in \bS^1$ with $\eta(\{x\}) = \eta(\{z\})$. Since $\eta$ is a probability measure, we see that $A_z$ is finite. Using the fact that every map in $\Gamma_{\nu}$ preserves $\eta$, we see that $A_z$ is $g$-invariant for every $g\in \Gamma_{\nu}$. Set 
    $$
        r=\min\{d(x,y)\colon\, x\neq y,\,x,y\in A_z\}.
    $$
    If $r>0$, then there exists $y\in A_z\backslash\{z\}$ and for every $g\in \Gamma_\nu$, we have $d(g(z),g(y))\geq r>0$, contradicting the proximality assumption. This shows that $A_z=\{z\}$ and by the invariance $g(z) = z$ for every $g\in \Gamma_{\nu}$.
    
    As a direct consequence of the above claim together with the proximality assumption, we conclude that $\eta$ has a single atom $z_0\in \bS^1$ satisfying $g(z_0) = z_0$ for every $g\in \Gamma_{\nu}$. To finish, we show that $\eta$ must coincide with $\delta_{z_0}$. If this is not the case, we can find $x\in \supp(\eta)\backslash \{z_0\}$. So, there exists $\varepsilon>0$ such that $z_0\notin B(x,\varepsilon)$ and 
    \begin{equation}\label{eq:cont-inv-p}
        \eta([x-\varepsilon,x])>0
        \quad
        \mbox{and}
        \quad
        \eta([x,x+\varepsilon])>0.
    \end{equation}
    Using proximality once more, we can find a sequence $(g_n)_n$ in $\Gamma_{\nu}$ with $g_n(x)\to z_0$ as $n\to\infty$. Without loss of generality assume that $g_n$ contracts the interval $[z,x]$. Since $g_n$ preserves $\eta$, we see that
    \begin{align*}
        \eta([z,x]) = \lim_{n\to\infty}\eta(g_n[z,x]) = \eta(\{z\}),
    \end{align*}
    which implies that $\eta((z,x]) = 0$. Since $[x-\varepsilon,x]\subset (z,x]$, we have $\eta([x-\varepsilon,x]) = 0$ contradicting \eqref{eq:cont-inv-p}. This proves that $\eta = \delta_{z_0}$.
\end{proof}

\begin{proposition}
\label{prop:noInvariantMeasure}
    If $\Gamma_{\nu}$ is proximal and does not fix any point in $\bS^1$, then there is no $\Gamma_{\nu}$-invariant measure on $\bS^1$.
\end{proposition}
\begin{proof}
    Direct consequence of the Proposition \ref{prop:invariantMeasureFixedPoint}. 
\end{proof}
\begin{remark}
\label{rk:noInvariantMeasure}
    Under the assumption of Proposition \ref{prop:noInvariantMeasure}, we can apply the Malicet's invariance principle and conclude that the exponent of contraction and the entropy of the RDS induced by $\nu$ are non zero. 
\end{remark}

Recall the definition of a $\nu$-stationary measure given in \eqref{eq:def-nu-stat}. For a random dynamical systems on the circle, stationary measures (measures satisfying \eqref{eq:def-nu-stat}) always exist. The unicity however, is not true in general, but in our setting we can
use the results in \cite{Mal:17} to conclude that the RDS induced by $\nu$ synchronizes orbits which ensures the unicity of the stationary measure. More precisely, we say that $\nu$ is \emph{synchronizing} if for every $x,y\in\bS^1$ and $\nu^{\bN}$-almost every $\omega\in X_\nu^\bN$,
\[
    \lim_{n\to\infty}d(f_\omega^n(x), f_\omega^n(y)) = 0.
\]
\begin{proposition}
\label{prop:sync+unicityStat}
    If $\Gamma_{\nu}$ is proximal and does not fix any point in $\bS^1$, then $\nu$ is synchronizing and admits a unique stationary measure $\eta$ on $\bS^1$. Moreover, $\eta$ is non-atomic.
\end{proposition}
\begin{proof}
     Using Proposition \ref{prop:noInvariantMeasure} and the results \cite[Theorem A and Proposition 4.18]{Mal:17}, we conclude that $\nu$ is synchronizing. Therefore, by dominated convergence theorem and \cite[Proposition 1]{Ste:2012}, we can conclude the uniqueness of the stationary measure without difficulty. Consider the $\nu$-stationary measure $\eta$.
     Following the reasoning in the proof of Proposition \ref{prop:invariantMeasureFixedPoint}, we observe that if $\eta$ has an atom, this atom must be fixed for every map in $\Gamma_{\nu}$ which contradicts our assumption.
\end{proof}
\begin{example}
\label{ex:proximalNonUniquelyErgodic}
    In general, proximal systems may have more than one ergodic stationary measure. For instance, consider the two matrices
\begin{align*}
    A = \begin{pmatrix}
        1 & 1 \\
        0 & 1
    \end{pmatrix}
    \quad
    \text{and}
    \quad
    B = \begin{pmatrix}
        1/2 & 0 \\
        0 & 2
    \end{pmatrix},
\end{align*}
and define the probability measure $\nu = p_1\delta_A + p_2\delta_B$ on $\hom(\bS^1)$, with $p_1,p_2>0$ (for the sake of this example, we identify $\bS^1$ with the projective space). Notice that the horizontal direction $x_0\in \bS^1$ is fixed by the $\Gamma_{\nu}  = \langle A, B\rangle^+$ action. Hence, the Dirac measure $\hat \eta = \delta_{x_0}$ is a $\nu$-stationary measure satisfying that
\[
    \iint \log\norm{gv}\, d\nu(g)d\hat \eta(v) = -p_2\log 2<0.
\]
However, the top Linear Lyapunov exponent of the linear cocycle generated by $\nu$ may be achieved by some $\nu$-stationary measure (see \cite[Proposition 6.7]{Via:14}, that is, there exists an ergodic $\nu$-stationary measure on $\bS^1$ such that
\[
    \iint \log\norm{gv}\, d\nu(g)d\eta(v) = p_2\log 2 >0.
\]
In particular, there is more than one $\nu$-stationary measures. 
\end{example}

The following lemma shows that synchronization for forward compositions also implies synchronization for reverse-order compositions. Although at finite time both forward and reverse compositions are equally distributed, we will later see that their asymptotic behavior differs: forward synchronization leads to accumulation points that fill the entire support of the stationary measure, while for the reverse compositions, for fixed 
$\omega$, the accumulation points reduce to a single point 
$\pi(\omega)$, which plays the role of an attractor in that setting.
\begin{lemma}
\label{lem:backward-from-forward}
Assume that for every $x,y\in\bS^1$,
\[
\nu^\bN\text{-almost every }\omega\in X_\nu^\bN:\quad \lim_{n\to\infty} d\!\left(f_\omega^n(x),f_\omega^n(y)\right)=0.
\]
Then for every $x,y\in\bS^1$,
\[
\nu^\bN\text{-almost every }\omega\in X_\nu^\bN:\quad \lim_{n\to\infty} d\!\left(\overline{f_\omega^n}(x),\overline{f_\omega^n}(y)\right)=0.
\]
\end{lemma}
\begin{proof}
Fix $x,y\in\bS^1$ and define
\[
    \cA:=\{\omega:\ \limsup_{n\to\infty} d\!\left(\overline{f_\omega^n}(x),\overline{f_\omega^n} (y)\right)>0\}.
\]
Our goal is to show that $\nu^\bN(\cA)=0$. Indeed, first note that
\[
    \cA=\bigcup_{q\in \bQ\cap(0,1]}\!\! \cA_q,
    \qquad
    \cA_q:=\{\omega:\ \LimS_{n\to\infty} d\!\left(\overline{f_\omega^n} (x),\overline{f_\omega^n}(y)\right)> q \}.
\]
Hence it suffices to prove $\nu^\bN(\cA_q)=0$ for each $q\in\bQ\cap(0,1]$. Fix $q\in \bQ\cap (0,1]$ and write
\[
    \overline{E}_n := \{\omega\colon\, d(\overline{f_\omega^n} (x),\overline{f_\omega^n} (y)) > q\}
    \quad
    \text{and}
    \quad
    E_n := \{\omega\colon\, d(f_\omega^n (x), f_\omega^n (y)) > q\}.
\]
Notice that for every $N\in\bN$, $\cup_{n\geq N}\overline{E}_n$ and $\cup_{n\geq N}\overline{E}_n$ are the limits of the increasing sequences $(\cup_{N\leq n\leq m} E_n)_{m\geq N}$ and $(\cup_{N\leq n\leq m} \overline{E}_n)_{m\geq N}$, respectively.

For each $\omega =(f_1,f_2,\ldots)\in \cup_{n = N}^m {E}_n$, set $n(\omega)=\min\{n\in\{N,\ldots,m\}\colon \omega\in E_n\}$ and define the sequence $\phi(\omega)=(f_{n(\omega)},f_{n(\omega)-1},\ldots,f_1,f_{n(\omega)+1},\ldots)$. Note that $\phi(\omega)\in \overline{E}_n\subset\cup_{n = N}^m \overline{E}_n$. Let $X_k := \{\omega\colon\, n(\omega) = k\}$ and notice that $\cup_{n=N}^m\phi(X_k) = \cup_{n=N}^m \overline{E}_n$. Hence, using that $\nu^{\N}$ is a product measure, we get
\[
    \nu^{\N}(\cup_{n=N}^m \overline{E}_n)
    = \nu^{\N}(\cup_{n=N}^m \phi(X_k))
    \leq \sum_{n=N}^m \nu^{\N}(\phi(X_k)) = \sum_{n=N}^m\nu^{\N}(X_k) = \nu^{\N}(\cup_{n=N}^m E_n).
\]
Analogously, the reverse inequality is proved, from which we conclude that
\[
    \nu^{\N}(\cup_{n = N}^m E_n)=\nu^{\N}(\cup_{n = N}^m \overline{E}_n).
\]

Then, by monotone convergence theorem,
\[
    \nu^{\N}(
        \cup_{n\geq N} {E}_n
    ) = \nu^{\N}(
        \cup_{n\geq N} \overline{E}_n
    ).
\]
Making $N\to \infty$, and using the assumption, we conclude that 
\[
    \nu^\bN(\cA_q)=\nu^\bN (\{
        \omega:\,
        \limsup_{n\to\infty}d\!\left(f_\omega^n(x),f_\omega^n(y)\right)> q
    \}) = 0,
\]
for every $q\in \bQ\cap (0,1]$. This concludes the proof of the lemma.
\end{proof}

By \cite[Theroem A]{Mal:17}, the absence of a probability measure which is simultaneously invariant by every map in $\Gamma_\nu$ implies local exponential contraction. More precisely, there exists $q\in (0,1)$ such that for every $x\in\bS^1$ and $\nu^\bN$-almost every $\omega\in X_\nu^\bN$ there exists an interval $I\subset\bS^1$ with $x\in\Int(I)$ such that
\[
\diam f_\omega^n(I)\leq q^n.
\]
In this case, we say that $\nu$ is \emph{locally contracting} with \emph{local contraction rate} $q$.

For the rest of this work, for $x,y\in \bS^1$, $x\neq y$, we write $[x,y]$ and $[y,x]$, for the two arcs connecting $x$ and $y$ in $\bS^1$.

The next Proposition is a refinement of \cite[Lemma 5.3]{GelSal:23}. It ensures not only synchronization with an exponential rate for points $x,y\in \bS^1$, but also that one of the arcs connecting $x$ and $y$, namely $[x,y]$ or $[y,x]$, will be contracted with the same rate. 
\begin{proposition}
\label{prop:syncofintervals}
    Assume that $\Gamma_{\nu}$ is proximal and does not fix any point in $\bS^1$. Let $q\in(0,1)$ be the associated local contraction rate. Then, for every $x,y\in\bS^1$ and $\nu^{\bN}$-almost every $\omega\in X_\nu^\bN$
    \[
        \lim_{n\to\infty}\frac 1n \log \diam f_\omega^n([x,y])\leq \log q
        \quad
        \mbox{or}
        \quad
        \lim_{n\to\infty}\frac 1n \log \diam f_\omega^n([y,x])\leq \log q.
    \]
\end{proposition}
\begin{proof}
    Let $\eta$ be the $\nu$-stationary measure on $\bS^1$.
    By \cite[Theorem 14.1]{Fur:73}, the product system $\omega\mapsto(f_\omega,f_\omega)$ on $\bS^1\times \bS^1$, admits an unique stationary measure $\mu$ on $\bS^1\times \bS^1$ such that
    \begin{equation}\label{eq01-lemma:syncofintervals}
        \mu(\{(x,x)\colon\, x\in\bS^1\})=1
        \quad
        \text{and}
        \quad
        \mu(\cdot\times \bS^1)=\eta(\cdot)=\mu(\bS^1\times \cdot).
    \end{equation}
    
    Consider the set $\mathcal{E}$ of points $(\omega,x,y)\in X_\nu^\bN\times\bS^1\times\bS^1$ satisfying the following conditions:
    \begin{enumerate}
        \item[(a)] there exist open neighborhoods $I_x$ of $x$ and $I_y$ of $y$ such that $\diam f_\omega^n(I_x)\to 0$ and $\diam f_\omega^n(I_y)\to 0$ as $n\to \infty$;
    
        \item[(b)] we have $\lim_{n\to\infty}\diam f_\omega^n([x,y])=0$ or $\lim_{n\to\infty}\diam f_\omega^n([y,x])= 0$.
    \end{enumerate}
    Note that if $(\omega,x,y)\in\mathcal{E}$ then $\{\omega\}\times I_x\times I_y\in \mathcal{E}$. So, we can write
    \[
        \mathcal{E}=\bigcup_{\omega\in X_\nu^\bN}
        (\{\omega\}\times U(\omega)),
    \]
    where $U(\omega)$ is an open subset of $\bS^1\times\bS^1$. Our goal is to apply \cite[Proposition 4.2]{Mal:17} with the set $\mathcal{E}$ and that it remains to check that $\nu^{\N}\otimes\mu(\mathcal{E}) > 0$. Notice that, by local contraction, for each $x\in \supp\eta$, we can find we can find an interval $I\subset \bS^1$ containing $x$ and a $\nu^{\N}$-full measure set $\Omega_I\subset X_{\nu}^{\N}$ such that if $\omega\in \Omega_I$, $\diam f^n_{\omega}(I)\to 0$ as $n\to\infty$. In particular, $\Omega_I\times I\times I\subset \mathcal{E}$ which implies that
    \[
        (\nu^\bN\otimes\mu)(\mathcal{E})
        \geq (\nu^\bN\otimes\mu)( \Omega_I\times I\times I) > 0.
    \]
    Hence, we can use \cite[Proposition 4.2]{Mal:17} and conclude that for every $x,y\in\bS^1$, we have $(\nu^\bN\otimes\delta_{(x,y)})(\mathcal{E})=1$. In particular, for every $x,y\in\bS^1$ and for $\nu^\bN$-almost every $\omega\in X_\nu^\bN$ 
    \[
        \lim_{n\to\infty}\diam f_\omega^n([x,y])=0 ,\quad\mbox{or}\quad \lim_{n\to\infty}\diam f_\omega^n([y,x])=0.
    \]
    The fact that the last convergence is exponential follows directly from the exponential synchronization as proved in \cite[Lemma 5.3]{GelSal:23}.
\end{proof}

\subsection{Random repelling point}
Let us state the following result where we first introduced a characterization of the map $\theta$ in Theorem \ref{teo:01}.
\begin{proposition}
\label{prop:NoSyncSection02}
    Assume that $\nu$ is proximal and that $\Gamma_\nu$ does not fix any point of $\bS^1$. Let $q\in(0,1)$ be the associated local contraction rate. Then there exists a measurable map $\theta\colon X_\nu^\bN\to\bS^1$ such that for $\mu$-almost every $\omega\in X_\nu^\bN$, for every interval $I\subset\bS^1\backslash\{\theta(\omega)\}$ and for every $x\neq\theta(\omega)$ we have
    \[
        \limsup_{n\to\infty}\frac{1}
        {n}\log \diam f^n_{\omega}(I) \leq \log q.
    \]
\end{proposition}
\begin{proof}
Consider a countable dense subset $Q$ of $\bS^1$.
By Proposition \ref{prop:syncofintervals}, for $\nu^\bN$-almost every $\omega\in X_\nu^\bN$ and every $x,y\in Q$
\[
    \limsup_{n\to\infty}\frac{1}
        {n}\log\diam f_\omega^n([x,y]) \leq \log q
        \quad
        \mbox{or}
        \quad
        \limsup_{n\to\infty}\frac{1}
        {n}\log\diam f_\omega^n([y,x])\leq \log q.
    \]
    Given $\omega\in X_\nu^\bN$ typical, fix $z\in Q$ and set
    \begin{align*}
        I_{+}&=\{x\in \bS^1\colon
            \limsup_{n\to\infty}\frac{1}
        {n}\log\diam f_\omega^n([z,x])\leq \log q
        \},
        \quad
        \mbox{and}\\
        I_{-}&=\{x\in \bS^1\colon
            \limsup_{n\to\infty}\frac{1}
        {n}\log\diam f_\omega^n([x,z])\leq \log q
        \}.
    \end{align*}
    Since $Q\subset I_{+}\cup I_{-}$ is dense, there exists an unique point $\theta(\omega)$ in $\bS^1\backslash Q$ such that 
    \[
    \theta(\omega)=\sup I_{+}=\inf I_{-}.
    \]
    Furthermore, if $x\neq \theta(\omega)$, we have $x\in I_{+}\cup I_{-}$.   
\end{proof}

Let $\Gamma_\nu$ be a semigroup on $\hom(\bS^1)$ generated by $X_\nu$. Here, we assuming that $\Id_{\bS^1}$ is an element of $\Gamma_\nu$.
\begin{corollary}\label{cor01:prop:NoSyncSection}
    Assume that $\Gamma_{\nu}$ is proximal and does not fix any point in $\bS^1$. Let $\eta$ be the $\nu$-stationary measure. Then there exists a sequence $(g_n)_{n\in\bN}$ in $\Gamma_\nu$ and a point $z\in \bS^1$ such that
    in weak${}^\ast$ topology
    \[
    \lim_{n\to\infty} (g_n)_\ast\eta=\delta_z.
    \]    
\end{corollary}
\begin{proof}
    Fix $\omega\in X_\nu^{\bN}$ typical such that the conclusion in Proposition \ref{prop:NoSyncSection02} holds. By Proposition \ref{prop:sync+unicityStat}, $\eta$ is non-atomic and hence $\eta(\{\theta(\omega)\})=0$. By compactness, we can find a subsequence $(n_k)_{k\in\bN}$ of natural numbers and $z\in \bS^1$, such that $f^{n_k}_{\omega}(x) \to z$ as $k\to\infty$, for all $x\neq\theta(\omega)$. We can also assume that the sequence of probability measures $((f_{\omega}^{n_k})_\ast\eta)_{k\in\bN}$ converges in the weak-$\ast$ topology. By dominated convergence, we can conclude that $(f_{\omega}^{n_k})_\ast\eta$ converges in the weak-$\ast$ topology to $\delta_z$ as $k\to \infty$. This concludes the prove of the corollary.
\end{proof}

\subsection{Random attracting point}
The next proposition ensures the existence of the point point $\pi$ on $\bS^1$ that plays the role of the unstable Oseledets direction. Recall that for $\omega = (f_j)_{j\in\N}\in X_{\nu}^{\N}$, we write $\overline{f^n_{\omega}} = f_1\circ f_2\circ\cdots\circ f_n$.
\begin{proposition}
\label{prop:NoSyncSection01}
    Assume that $\Gamma_{\nu}$ is proximal and does not fix any point in $\bS^1$. Let $\eta$ be the $\nu$-stationary measure. Then there exists a measurable map $\pi: X_\nu^\bN\to\bS^1$ such that $\nu^\bN$-almost every $\omega\in X_\nu^\bN$ we have
    \[
        \lim_{n\to\infty}\overline{f_{\omega}^n}_\ast \eta=\delta_{\pi(\omega)}.
    \]
\end{proposition}
\begin{proof}
    Let $\eta$ be the $\nu$-stationary measure. By compactness of $\bS^1$ and the continuity of the maps in $X_\nu$ we can apply \cite[Proposition 1]{GuiRau:86} to get that  $\nu^{\bN}$-almost every $\omega\in X_\nu^\bN$ there exists a probability measure $\rho_\omega$ on $\bS^1$
    \begin{equation}\label{eq:00001cor}
    \rho_\omega=\lim_{n\to\infty}\overline{f_{\omega}^n}_\ast \eta=\lim_{n\to\infty}(\overline{f_{\omega}^n}\circ f_\xi^\ell)_\ast \eta,
    \end{equation}
    for $\nu^\bN$-almost every $\xi\in X_\nu^\bN$ and all $\ell\in\bN$. Where the convergence above is in the weak${}^\ast$-topology of the space of probability measures on $\bS^1$. The measurability of the map $\omega\mapsto\rho_\omega$ was also established in \cite{GuiRau:86}.
    
    Fix $\omega\in X_\nu^\bN$ such that \eqref{eq:00001cor} holds. Let us show that $\rho_\omega$ is a Dirac measure. Let $\hom^+(\bS^1)$ and $\hom^-(\bS^1)$ be the set of orientation-preserving and orientation-reversing homeomorphisms of the circle, respectively. Since $\hom(\bS^1)=\hom^+(\bS^1)\cup\hom^-(\bS^1)$, we can find a subsequence $(n_k)_{k\in\bN}$ of natural numbers such that each $\overline{f_{\omega}^{n_k}}$, $k\in\bN$, is in the same subset $\hom^\ast(\bS^1)$ simultaneously, for some $\ast\in\{+,-\}$. Let $Q=\{q_1,q_2,\cdots\}$ be a dense, countable subset of $\bS^1$. Consider a subsequence $(n_k(1))_{k\in\bN}$ of $(n_k)_{k\in\bN}$ such that $(\overline{f_{\omega}^{n_k(1)}}(q_1))_{k\to\infty}$ is convergent. For $\ell>1$, consider a subsequence $(n_k(\ell))_{k\in\bN}$ of $(n_k(\ell-1))_{k\in\bN}$ such that $(\overline{f_{\omega}^{n_k(\ell)}}(q_\ell))_{k\to\infty}$ is convergent. Therefore, we find a subsequence $(n_k')_{k\in\bN}$ of $(n_k)_{k\in\bN}$ such that for all $q\in Q$, the sequence $(\overline{f_{\omega}^{n_k'}}(q))_{k\to\infty}$ is convergent. Since for some $\ast\in\{+,-\}$, each map $\overline{f_{\omega}^{n_k'}}\in\hom^\ast(\bS^1)$, we get the set of the points $x\in\bS^1$ such that the sequence $(\overline{f_{\omega}^{n_k'}}(x))_{k\to\infty}$ does not converge is countable. Using the same argument above for subsequence extraction, we find a subsequence $(n_k'')_{k\in\bN}$ of $(n_k')_{k\in\bN}$ such that for all $x\in\bS^1$ the sequence $(\overline{f_{\omega}^{n_k''}}(x))_{k\to\infty}$ is convergent.
    Hence, the measurable map $h:\bS^1\to\bS^1$ given by
    \[
    h=\lim_{k\to\infty}\overline{f_{\omega}^{n_k''}}
    \]
    has at most countably many discontinuity points. In particular, the set of discontinuities of $h$ has $\eta$-measure zero.
    Then,  for $\nu^\bN$-almost every $\xi\in X_\nu^\bN$ and all $\ell\in\bN$.
    \begin{equation}\label{eq:00002cor}
    \rho_\omega=h_\ast \eta=(h\circ f_\xi^\ell)_\ast\eta.
    \end{equation}
    Since $X_\nu$ is the topological support of $\nu$, we get that \eqref{eq:00002cor} holds for all $\xi\in X_\nu^\bN$. That is,
    \begin{equation}\label{eq:00003cor}
    \rho_\omega=h_\ast \eta=(h\circ g)_\ast\eta,
    \end{equation}
    for all $g\in\Gamma_\nu$. By Corollary \ref{cor01:prop:NoSyncSection}, there exists a sequence $(g_n)_{n\in\bN}$ in $\Gamma_\nu$ and a point $z\in \bS^1$ such that
    \begin{equation}\label{eq:00004cor}
    \lim_{n\to\infty} (g_n)_\ast\eta=\delta_z.
    \end{equation} 
    Therefore, from \eqref{eq:00003cor} and \eqref{eq:00004cor}, we get that $\rho_\omega$ is a Dirac measure at $h(z)$. 
    To conlude this proof, define the measurable map $\pi:X_\nu^\bN\to\bS^1$ such that $\rho_\omega=\delta_{\pi(\omega)}$. 
\end{proof}
\begin{remark}
    Under the assumption of contraction on average, \cite[Theorem 1]{Ste:2012} already gives the existence of the random point $\pi$ in Proposition \ref{prop:NoSyncSection01}. When $\nu$ is proximal and supported in the group of circle diffeomorphisms with the semigroup $\Gamma_{\nu}$ not fixing any point in $\bS^1$, we may use \cite[Theorem 1.1]{GelSal:23} to conclude that the RDSs defined by $\nu$ contracts on average, obtaining a simple proof of the existence of the random point in the differentiable case.
\end{remark}

\begin{proposition}
\label{prop:pointwiseConvergenceToPi}
Assume that $\Gamma_{\nu}$ is proximal and does not fix any point in $\bS^1$. Let $\pi: X_\nu^\bN\to\bS^1$ be the measurable map defined in Proposition \ref{prop:NoSyncSection01}. Then
for every \(x \in \bS^1\) and $\nu^\bN$-almost every $\omega\in X_\nu^\bN$,
\[
    \lim_{n\to\infty}\overline{f_\omega^n}(x) = \pi(\omega).
\]
\end{proposition}

\begin{proof}
Fix \(x \in \bS^1\). Suppose, by contradiction, that there exists a set 
\(\Omega_x \subset \Omega\) with \(\nu^\bN(\Omega_x) > 0\) such that for every $\omega\in \Omega_x$, $\overline{f_\omega^n}(x)$ does not converge to $\pi(\omega)$.

Let \(Q \subset \bS^1\) be a fixed countable dense subset, and define
\[
    \Omega_{Q} := \bigcap_{q\in Q}\{
        \omega\in \Omega\colon\,
        \lim_{n \to \infty} d(\overline{f_\omega^n}(q),\overline{f_\omega^n}(x))=0
    \}.
\]
By Proposition \ref{prop:sync+unicityStat} and Lemma \ref{lem:backward-from-forward}, \(\nu^\bN(\Omega_Q) = 1\). Also, define the full $\nu^{\N}$-measure set
\[
    \Omega^\ast := \{\omega \in \Omega : \pi(\omega) \text{ exists}\}.
\]
Note that $\nu^\bN(\Omega_x \cap \Omega_Q \cap \Omega^\ast )=\nu^\bN(\Omega_x  )>0$. Fix $\omega \in \Omega_x \cap \Omega_Q \cap \Omega^\ast$.

By compactness, there exists a subsequence \((n_k)_k\subset \N\) and a point \(p \in \bS^1 \setminus \{\pi(\omega)\}\) such that $\overline{f_\omega^{n_k}}(x) \to p$. Moreover, since $\omega\in\Omega_Q$, for every $q\in Q$, $\overline{f_\omega^{n_k}}(q) \to p $. 

In particular, for any pair of points $q_1,q_2\in Q$, we can find a subsequence $(n_k)_{k\in \N'}$ of $(n_k)$ such that either $\diam f^{n_k}_{\omega}([q_1,q_2])\to 0$ or $\diam f^{n_k}_{\omega}([q_2,q_1])\to 0$ as $k\in \N'$ goes to infinity. Hence, using the density of $Q$ and a diagonal argument, we have that, up to passing to a subsequence, there exists a point $t_p\in\bS^1$ such that for all $y\in\bS^1\backslash\{t_p\}$
\begin{equation}
\overline{f_\omega^{n_k}}(y) \longrightarrow p.
\end{equation}
Therefore, for every $\varphi\in C^0(\bS^1)$, 
\[
    \int \varphi(\bar{f}^{n_k}_{\omega}(y))\, d\eta(y) \to \varphi(p).
\]
Now using Proposition \ref{prop:NoSyncSection01}, we see that $\varphi(p) = \varphi(\pi(\omega))$, for every $\varphi\in C^0(\bS^1)$. We conclude that, $p = \pi(\omega)$, which is a contradiction.
\end{proof}

\subsection{Proofs of  Theorem \ref{teo:01} and Corollary \ref{cor01:teo01}}
\begin{proof}[Proof of  Theorem \ref{teo:01}]
    Consider $\nu$ to be proximal so that its support $X_\nu$ does not fix any point of $\bS^1$. By Proposition \ref{prop:proxInverseSystem}, $\nu^-$ is also proximal. Then applying Proposition \ref{prop:NoSyncSection01} and Proposition \ref{prop:NoSyncSection02} to $\nu$ and $\nu^-$ there exist measurable maps $\theta^+,\pi^+,\theta^-,\pi^-:X_\nu^\bN\to\bS^1$ such that 
    \begin{enumerate}
         \item for all $x\in\bS^1$ and  $\nu^\bN$-almost every $\omega=(f_n)_{n\in\bN}\in X_\nu^\bN$
         $$
         \lim_{n\to\infty}f_1\circ\cdots\circ f_n(x)=\pi^+(\omega)
         \quad\mbox{and}\quad
         \lim_{n\to\infty}f_1^{-1}\circ\cdots\circ f_n^{-1}(x)=\pi^-(\omega);
         $$
         \item  for $\nu^\bN$-almost every $\omega=(f_n)_{n\in\bN}\in X_\nu^\bN$ and for each pair of closed sets $A\subset\bS^1\backslash\{\theta^+(\omega)\}$, $B\subset\bS^1\backslash\{\theta^-(\omega)\}$,
         $$
         \lim_{n\to\infty}\diam f_n\circ\cdots\circ f_1(A)=
         \lim_{n\to\infty}\diam f_n^{-1}\circ\cdots\circ f_1^{-1}(B)=0.
         $$
     \end{enumerate}
     To conclude the first part of the proof, it is enough to see that $\theta^+=\pi^-$ and $\theta^-=\pi^+$.
     
    Let $Q\subset \bS^1$ be a countable dense set. Applying Proposition \ref{prop:pointwiseConvergenceToPi} for $\nu^-$, we can find a set $ \Omega^\ast\subset X_{\nu}^{\N}$, with $\nu^{\N}(\Omega^\ast) = 1$, such that for every $\omega = (f_n)_{n\in \N}\in \Omega^{\ast}$
    \[
        (f^n_{\omega})^{-1}(q) = f_1^{-1}\circ\cdots \circ f^{-1}_n(q) \to \pi^-(\omega),
        \quad
        \text{for every }
        q\in Q.
    \]
    Let $I\subset \bS^1\backslash\{\theta^+(\omega)\}$ be any closed interval. Since $\diam f^n_{\omega}(I)\to 0$ as $n\to \infty$, we can find a subsequence $(n_k)_{k\in\N}\subset \N$, such that $f^{n_k}_{\omega}(I)$ converges to a point $p\in \bS^1$. In particular, we can find a point $q\in Q$ and $k_0\in \N$, such that $q\notin f^{n_k}_{\omega}(I)$, for every $k\geq k_0$. Hence, $(f^{n_k}_{\omega})^{-1}(q) \notin I$, for every $k\geq k_0$. This implies that $\pi^-(\omega)\notin I$. Since the choice of the interval $I\subset \bS^1\backslash\{\theta(\omega^+)\}$ was arbitrary, we conclude that $\pi^-(\omega) = \theta^+(\omega)$. The proof that $\theta^- = \pi^+$ is analogous.
     
     
     To conclude that the convergences in (1) and in (2) are exponentially fast with rate $q\in(0,1)$ (the rate of local contraction) use Proposition \ref{prop:syncofintervals} and recall that for all $n\in\bN$, random variables $\omega\mapsto f^n_{\omega}$ and $\omega\mapsto \overline{f_{\omega}^n}$ have the same distribution.

    To finish, note that for $\nu^\bN$-almost every $\omega\in X_\nu^\bN$ we have that for all $n\in\bN$
    \[
    \pi(\omega)=\overline{f_\omega^n}(\pi(\sigma^n\omega))\quad\mbox{and}\quad f_\omega^n(\theta(\omega))=\theta(\sigma^n(\omega)).
    \]
\end{proof}

\begin{proof}[Proof of Corollary \ref{cor01:teo01}]
    Under the conditions of $\nu$, we know that there exists a unique $\nu$-stationary measure $\eta$ and a unique $\nu^-$-stationary measure $\eta^-$. What we have to prove is 
    \[
     \eta=\int_{X_\nu^{\bN}}\delta_{\pi(\omega)}\,\dd\nu^{\bN}\quad\mbox{and}\quad\eta^-=\int_{X_\nu^{\bN}}\delta_{\theta(\omega)}\,\dd\nu^{\bN},
     \]
     Using that $\theta = \theta^+ = \pi^-$, it is enough to prove the first equality. By $\nu$-stationarity,
    $$
        \eta
        =\int_{X_\nu^{\bN}}(f_\omega^n)_\ast\eta\,\dd\nu^\bN(\omega)
        =\int_{X_\nu^{\bN}}(\overline{f_\omega^n})_\ast\eta\,\dd\nu^\bN(\omega),
    $$
    for all $n\in\bN$. Using Proposition \ref{prop:NoSyncSection01},
    \begin{align*}
        \eta&
        =\lim_{n\to\infty}\int_{X_\nu^{\bN}}(\overline{f_\omega^n})_\ast\eta\,\dd\nu^\bN(\omega)
        =\int_{X_\nu^{\bN}}\lim_{n\to\infty}(\overline{f_\omega^n})_\ast\eta\,\dd\nu^\bN(\omega)
        =\int_{X_\nu^{\bN}}\delta_{\pi(\omega)}\,\dd\nu^\bN(\omega).
    \end{align*}
    The corollary is proven.   
\end{proof}

\section{Differentiable case}
\label{sec:diffCase}
In this section, we prove Theorems \ref{teo:02} and \ref{teo03}. Throughout this section, we assume that $\nu$ is a probability measure on $\diff^{1+\tau}(\bS^1)$. Also, we assume that $\nu$ is proximal, \eqref{int-cond} holds and $X_\nu$ does not fix any point in $\bS^1$.

For any function $f\in \diff^{1+\tau}(\bS^1)$, let $\alpha_f:(0,\infty)\to\R$ be the \emph{modulus of continuity} of the function $\log|f'|$, that is, for $\varepsilon>0$,
\begin{align*}
    \alpha_f(\varepsilon)
    := \sup_{d(x,y)\leq \varepsilon}\,
    \left|
        \log\frac{|f'(x)|}{|f'(y)|}
    \right|.
\end{align*}
The Hölder continuity of $f'$ implies that $\alpha_f(\varepsilon)\to 0$ as $\varepsilon\to 0^+$. Further, the integrability condition \eqref{int-cond} implies that
\[
\int_{X_\nu}\alpha_f(\varepsilon)\,\dd\nu(f)<\infty,
\]
for all $\varepsilon>0$.
\begin{lemma}
\label{lem:modcont}
    For $\nu^{\bN}$-almost every $(f_n)_{n\in\bN}\in X_\nu^\bN$ we have that 
    \[
    \lim_{n\to\infty}\frac{1}{n}\sum_{j=1}^{n}\alpha_{f_j}(\varepsilon_j)=0,
    \]
    for any sequence of positive real numbers $(\varepsilon_n)_{n\in \N}$ converging to $0$.
\end{lemma}
\begin{proof}
    Applying Birkhoff's theorem, we see that, for every $k\geq 1$ and $\nu^{\N}$-almost every $(f_n)_{n\in\N}\in X^{\N}_{\nu}$, the sequence of averages $\frac{1}{n}\sum_{j=1}^{n}\alpha_{f_j}(1/k)$ converges to the integral of $f\mapsto\alpha_f(1/k)$ with respect to $\nu$, as $n\to\infty$. The result follows from the monotonicity of $\varepsilon\mapsto \alpha_f(\varepsilon)$ and dominated convergence theorem.
\end{proof}

The next Lemma provides a control of distortion, which is uniform far from the point $\theta(\omega)\in \bS^1$, where $\theta:X_\nu^\bN\to \bS^1$ is the measurable map given by Proposition \ref{prop:NoSyncSection02}.
\begin{lemma}
\label{lem:distortionControl}
Given $\delta>0$, for $\nu^{\N}$-almost every sequence $\omega \in X^{\N}_{\nu}$, and for every closed interval $I\subset \bS^1\backslash\{\theta(\omega)\}$, there exists $N\in \N$ such that for every $n\geq N$ and every $x,y\in I$,
\[
    e^{-n\delta}\leq \frac{|(f^n_{\omega})'(x)|}{|(f^n_{\omega})'(y)|} \leq e^{n\delta}.
\]
\end{lemma}
\begin{proof}
Let $\omega\in X^{\N}_{\nu}$ be an element in the full measure set given by Proposition \ref{prop:NoSyncSection02} and define $\varepsilon_j := \diam(f^j_{\omega}(I))$, $j\geq 1$. Applying Lemma \ref{lem:modcont} with the sequence $(\varepsilon_j)_{j\geq 1}$, the result follows.
\end{proof}

Let $\eta$ be the $\nu$-stationary probability measure. Set
\begin{align*}
    \lambda(\eta,\nu) := \int \log|f'(x)|d(\eta\otimes\nu)(x,f).
\end{align*}
For $\omega \in X^{\N}_{\nu}$, we define
\begin{align*}
    \lambda(\omega,x)
    := \lim_{n\to\infty}\frac{1}{n}\log|(f^n_{\omega})'(x)|
    = \lambda(\eta,\nu),
\end{align*}
if the above limit exist. In particular, by Birkhoff's ergodic theorem, for $\nu^{\N}\otimes\eta$-almost every $(\omega,x)\in  X_\nu^\bN\times\bS^1$ we have $\lambda(\omega,x) = \lambda(\eta,\nu)$.

\begin{lemma}
\label{lem:lox=leta}
    Let $\nu$ be a probability measure on $\diff^{1+\tau}(\bS^1)$ satisfying the integrability condition \eqref{int-cond}. Then, for $\nu^\bN$-almost every $\omega\in X_\nu^\bN$ and any $x\neq\theta(\omega)$ we have 
    \[
        \lambda(\omega,x)=\lambda(\eta,\nu).
    \]
\end{lemma}
\begin{proof}
Let $\omega$ be a $\nu^\bN$-typical sequence in $X_\nu^\bN$. Since the $\nu$-stationary measure $\eta$ is non-atomic and by typicality of $\omega$, we can find some $y\in \bS^1\backslash\{\theta(\omega)\}$ such that $\lambda(\omega,y)=\lambda(\eta,\nu)$. Now take any $x\in \bS^1\backslash\{\theta(\omega)\}$. The result is now an immediate consequence of Lemma \ref{lem:distortionControl} applied with $I\subset \bS^1\backslash\{\theta(\omega)\}$ being the arc connecting $x$ and $y$.
\end{proof}
\subsection{Extremal Lyapunov exponents}
By Kingman's subadditive ergodic theorem, for $\nu^\bN$-almost every $\omega\in X_\nu^\bN$
$$
    \Lambda(\nu)
    = \lim_{n\to\infty}\frac{1}{n}\log\sup_{z\in \bS^1}|(f^n_{\omega})'(z)|
    \quad
    \text{and}
    \quad
    \lambda(\nu)
    = \lim_{n\to\infty}\frac{1}{n}\log\inf_{z\in \bS^1}|(f^n_{\omega})'(z)|.
$$
For every $\omega\in X_\nu^\bN$ and $n\geq 1$, let
\begin{align*}
    M_n(\omega) := \left\{
        z\in\bS^1\colon\, |(f^n_{\omega})'(z)| = \sup_{x\in\bS^1}|(f^n_{\omega})'(x)|
    \right\}.
\end{align*}

This is a non-empty compact subset of $\bS^1$.  The next proposition ensures that for typical $\omega\in X^{\N}_{\nu}$, the points where $|(f^n_{\omega})'|$ attains its maximum value accumulate at $\theta(\omega)$.
\begin{proposition}
\label{prop:exp-max-theta}
    For $\nu^\bN$-almost every $\omega\in X_\nu^\bN$ and any sequence $(z_n(\omega))_{n\in \N}$ with $z_n(\omega)\in M_n(\omega)$ we have that $z_n(\omega)\to \theta(\omega)$.
\end{proposition}
\begin{proof}
    Let $\omega\in X^{\N}_{\nu}$ be a $\nu^{\bN}$-typical sequence. Assume, by contradiction, that there exists a subsequence $(z_{n_k}(\omega))_k$, $z_{n_k}(\omega)\in M_{n_k}(\omega)$, that converges to some $z\neq \theta(\omega)$. Then, for some closed interval $I\subset\bS^1\backslash\{\theta(\omega)\}$, there exists $N\in\bN$ such that for $k\geq N$, $z_{n_k}(\omega)\in I$.
    
    By Lemma \ref{lem:lox=leta} and Lemma \ref{lem:distortionControl}, we can assume that $N\geq 1$ is large enough  such that $|(f_\omega^n)'(x)|<1$, for every $x\in I$ and $n\geq N$. Hence, for $k$ large enough
    \[
        |(f_\omega^{n_k})'(z_{n_k}(\omega))|<1.
    \]
    However, for every $k\geq 1$, $|(f_\omega^{n_k})'(z_{n_k}(\omega))|=\sup_{x\in\bS^1}|(f_\omega^{n_k})'(x)|\geq 1$.
\end{proof}

\begin{proposition}
\label{prop:maxExponentIsThetaGrowth}
For $\nu^{\N}$-a.e. $\omega\in \Omega$,
\begin{align*}
    \Lambda(\nu) = \lim_{n\to\infty}\frac{1}{n}\log|(f^n_{\omega})'(\theta(\omega))| = \lambda(\omega,\theta(\omega)).
\end{align*}
\end{proposition}
\begin{proof}
    Fixing the Hausdorff topology on compact subsets of $\bS^1$, we see that for every $n\geq 1$, the function $\Omega\ni \omega\mapsto M_n(\omega)$ is measurable. In particular, we can find a measurable selection $z_n:\Omega\to \bS^1$, where for every $\omega\in \Omega$, $z_n(\omega)\in M_n(\omega)$ (\cite[Proposition 4.6]{Via:14}). For each $n\geq 1$, define the measure $m_n$ on $\Omega\times\bS^1$ by
    \begin{align*}
        m_n(\Gamma\times C) = \frac{1}{n}\sum_{j=0}^{n-1}\int_{\Gamma}\, \delta_{f^j_{\omega}(z_n(\omega))}(C)\, d\nu^{\N},
    \end{align*}
    for any $\Gamma\times C\subset \Omega\times\bS^1$ measurable. Let $m$ be a limit of a subsequence of $(m_n)_n$, say $(m_{n_k})_k$, and notice that
    \begin{align*}
        \int_{\Omega\times\bS^1} \log|f_{\omega}'(z)|\, dm(\omega,z)
        &= \lim_{k\to\infty}\int_{\Omega\times\bS^1} \log|f_{\omega}'(z)|\, dm_{n_k}(\omega,z)\\
        &= \lim_{k\to\infty}\frac{1}{n_k}\int_{\Omega}\log|(f^{n_k}_{\omega})'(z_{n_k}(\omega))|\, d\nu^{\N}\\
        &= \Lambda(\nu).
    \end{align*}
    Also, by Birkohff ergodic theorem, there exist a measurable function $\lambda:\Omega\times \bS^1\to \R$ and set $X\subset \Omega\times\bS^1$, with $ m(X) = 1$ such that for every $(\omega,z)\in X$, 
    \begin{align*}
        \lim_{n\to\infty}\frac{1}{n}\log|(f^n_{\omega})'(z)| = \lambda(\omega,z)
        \quad
        \text{and}
        \quad
        \int \lambda\, dm = \int \log|f_{\omega}'(z)|\, dm(\omega,z) = \Lambda(\nu).
    \end{align*}
    Let $\Omega_0\subset \Omega$, $\nu^{\N}(\Omega_0) = 1$, such that for $\omega\in \Omega_0$,
    \begin{align*}
        \lim_{n\to\infty}\frac{1}{n}\log\sup_{z\in \bS^1}|(f^n_{\omega})'(z)| = \Lambda(\nu).
    \end{align*}
    Hence, for every $(\omega,z)\in X\cap (\Omega_0\times\bS^1)$,
    \begin{align*}
        \int\, \lambda\, dm = \Lambda
        \quad
        \text{and}
        \quad
        \lambda(\omega,z) \leq \Lambda(\nu)
        \text{ for every }
        (\omega,z)\in X\cap (\Omega_0\times\bS^1)
    \end{align*}
    This implies that the set $X_0 = \{(\omega,z)\in \Omega\times\bS^1\colon \lambda(\omega,z) = \Lambda(\nu)\}$ has total $m$-measure. Let $\pi:\Omega\times\bS^1\to \Omega$ be the projection in the first coordinate and notice that $\pi_*\, m = \nu^{\N}$. Then, the set $\pi(X_0)$ is measurable (\cite[Proposition 4.5]{Via:14}) and $\nu^{\N}(\pi(X_0)) = 1$.
    
    For any $\omega\in \pi(X_0)$, there exists $z(\omega)$ such that
    \begin{align*}
        \lim_{n\to\infty}\frac{1}{n}\log|(f^n_{\omega})'(z(\omega))| = \lambda(\omega,z) = \Lambda(\nu).
    \end{align*}
    Observe that, by Lemma \ref{lem:lox=leta}, $z(\omega)$ must coincide with $\theta(\omega)$. This finishes the proof of the proposition.
\end{proof}

Consider the probability measure $\nu^-$ supported in $X_\nu^-:=\{f^{-1}:f\in X_\nu\}$ as follows
\begin{equation}\label{def:eta-}
    \nu^-(\cdot)=\nu(\{f:f^{-1}\in\cdot\}).
\end{equation}
Note that $(\nu^-)^-=\nu$. If $\nu$ is proximal, it not true in general that $\nu^-$ is proximal. However, if we assume additionally that $\Gamma_{\nu}$ does not fix any point in $\bS^1$, then $\nu^-$ is proximal and this is the content of the next proposition.
\begin{proposition}
\label{prop:proxInverseSystem}
    Assume that $\Gamma_{\nu}$ is proximal and does not fix any point in $\bS^1$. Then, $\Gamma_{\nu^-}$ is also proximal and does not fix any point in $\bS^1$.
\end{proposition}
\begin{proof}
    Note that a point is fixed by every element in $\Gamma_\nu$, if and only if, it is also fixed by every element in $\Gamma_\nu^-$.
    
    Let $x,y\in\bS^1$. Let $\theta\colon X_\nu^\bN\to\bS^1$ be as in Proposition \ref{prop:NoSyncSection02}. The property of synchronization implies that for $\nu^\bN$-almost every $\omega$, $x\neq\theta(\omega)$ and $y\neq\theta(\omega)$. Fix a typical sequence $\omega=(f_n)_{n\in\bN}\in X_\nu^\bN$. By Proposition \ref{prop:NoSyncSection02}, for all $k\in\bN$  we can find a natural number $n_k$, a positive number $r_k>0$ and a point $z_k\neq x$ such that 
    \begin{equation}\label{eq:theta-inverse-prox}
    f_\omega^{n_k}([\theta(\omega)+1/k,\theta(\omega)-1/k])\subset [z_k-r_k,z_k+r_k],\quad\mbox{and}\quad x,y\notin[z_k-r_k,z_k+r_k]
    \end{equation}
    Above, we have used that the $\nu$-stationary measure $\eta$ is non-atomic and that the orbits $(f_\omega^n(z))_{n\in\bN}$ must accumulate in the entire support of $\eta$. In our notation, $z\in [z-r,z+r]$ and $\bS^1=[z-r,z+r]\cup[z+r,z-r]$ for all $z\in\bS^1$ and $r>0$. From \eqref{eq:theta-inverse-prox}, we must have
    \[
    (f_\omega^{n_k})^{-1}([z_k+r_k,z_k-r_k])\subset [\theta-1/k,\theta+1/k],
    \]
    and so, $(f_\omega^{n_k})^{-1}(x),(f_\omega^{n_k})^{-1}(y)\in[\theta-1/k,\theta+1/k]$. To conclude, note that for all $k\in\bN$, $(f_\omega^{n_k})^{-1}$ is an element of the semigroup generated by the topological support of $\nu^-$.
\end{proof}
\begin{remark}
    We emphasize that we do not make any use of the differentiability assumption in the proof of Proposition \ref{prop:proxInverseSystem}.
\end{remark}

Since $\Gamma_{\nu}$ is proximal and does not fix any point in $\bS^1$, Proposition \ref{prop:proxInverseSystem} and Proposition \ref{prop:sync+unicityStat}, ensures that $\nu^-$ admits an unique $\nu^-$-stationary measure $\eta^-$ on $\bS^1$.
\begin{proposition}
\label{prop:expintheta-expinpi}
    For $\nu^{\N}$-almost every $\omega=(f_n)_{n\in\bN}\in X_\nu^\bN$,
    \begin{align*}
        \lim_{n\to\infty}\frac{1}{n}\log|(f_1\circ\cdots\circ f_n)'(\pi(\sigma^n(\omega)))|=\lim_{n\to\infty}\frac{1}{n}\log|(\overline{f_{\omega}^n})'(\pi(\sigma^n(\omega)))| = \lambda(\eta,\nu),
    \end{align*}
    and 
    \[
    \lim_{n\to\infty}\frac{1}{n}\log|(f_n\circ\cdots\circ f_1)'(\theta(\omega))|=\lim_{n\to\infty}\frac{1}{n}\log|(f_\omega^n)'(\theta(\omega))|= -\lambda(\eta^-,\nu^-).
    \]
    \end{proposition}
    \begin{proof}
        Consider the function $\varphi:X_\nu^\bN\to \R$ defined, for every $\omega = (f_n)_{n\in\N}$, by, $\varphi(\omega) = \log|f_1'(\pi(\sigma\omega))|$. Notice that,
        \begin{align*}
            \log|(\overline{f_{\omega}^n})'(\pi(\sigma^n\omega)))|
            = \sum_{j=1}^n \log|f_j'(\overline{f^{n}_{\sigma^j\omega}}(\pi(\sigma^n\omega)))|
            = \sum_{j=0}^{n-1}\varphi(\sigma^j(\omega)).
        \end{align*}
        Hence, by Birkhoff ergodic theorem, for $\nu^{\N}$-almost every $\omega = (f_n)_{n\in\N}\in X_\nu^\bN$,
        \begin{align*}
            \lim_{n\to\infty}\frac{1}{n}\log|(\overline{f_{\omega}^n})'(\pi(\sigma^n(\omega)))|
            &= \int_{X_\nu^\bN} \log|f_{1}'(\pi(\sigma\omega)))|\,\dd\nu^{\N}(\omega)\\
            &= \int_{X_\nu}\int_{\bS^1} \log|f'(x)|\,\dd\eta(x)\,\dd\nu(f)
            = \lambda(\eta,\nu).
        \end{align*}
    In the last, we used that $\nu^\bN$ is $\sigma$-invariant.

    On the other hand, note that for all $n\in\bN$, using what was obtained above for $\nu^-$ (instead of $\nu$), we get $\nu^{\N}$-almost every $\omega=(f_n)_{n\in\bN}\in X_\nu^\bN$
    \begin{align*}
        \lim_{n\to\infty}\frac{1}{n}\log|((f_\omega^n)^{-1})'(\theta(\sigma^n(\omega)))|
        &=\lim_{n\to\infty}\frac{1}{n}\log|(f_1^{-1}\circ\cdots\circ f^{-1}_n)'(\theta(\sigma^n(\omega)))|\\
        &= \lambda(\eta^-,\nu^-).
    \end{align*}
    Since 
    \[
    \theta(\sigma^n\omega)=f_n\circ\cdots\circ f_1(\theta(\omega))=f_\omega^n(\theta(\omega)),
    \]
    we conclude
    \[
        \lim_{n\to\infty}\frac{1}{n}\log|(f_\omega^n)'(\theta(\omega))|=-\lambda(\eta^-,\nu^-)
    \]
     \end{proof}

\subsection{Proof of Theorem \ref{teo:02} and Corollary \ref{cor:teo:02}}
We first use the tools established in this section to give the characterization of the extremal Lyapunov exponents $\Lambda(\nu)$ and $\lambda(\nu)$.

Applying Proposition \ref{prop:maxExponentIsThetaGrowth} and Proposition  \ref{prop:expintheta-expinpi} for $\nu$ and $\nu^-$, we see that $\Lambda(\nu) = - \lambda(\eta^-,\nu^-)$ and $\Lambda(\nu^-) = - \lambda(\eta,\nu)$. However, note that
\begin{align*}
    \Lambda(\nu^-)&=\lim_{n\to\infty}\frac{1}{n}\int_{ X_\nu^\bN} \log\sup_{z\in \bS^1}|(f_n^{-1}\circ\cdots\circ f_1^{-1})'(z)|\,\dd\nu^{\N}\left((f_n)_{n\in\bN}\right)\\
    &=-\lim_{n\to\infty}\frac{1}{n}\int_{ X_\nu^\bN} \log\inf_{z\in \bS^1}|(\overline{f_\omega^n})'(z)|\,\dd\nu^{\N}(\omega)\\
    &=-\lim_{n\to\infty}\frac{1}{n}\int_{ X_\nu^\bN} \log\inf_{z\in \bS^1}|(f_\omega^n)'(z)|\,\dd\nu^{\N}(\omega)=-\lambda(\nu),
\end{align*}
where we used that the random variables $\omega\mapsto \inf_{z\in \bS^1}|(\overline{f_\omega^n})'(z)|$ and $\omega\mapsto \inf_{z\in \bS^1}|(f_\omega^n)'(z)|$ have the same distribution. Therefore,
\begin{align}
\label{eq:160825.1}
    \Lambda(\nu) = - \lambda(\eta^-,\nu^-)
    \quad
    \text{and}
    \quad
    \lambda(\nu) = \lambda(\eta,\nu).
\end{align}
We are now in position to finish the proof of Theorem \ref{teo:02}. Proposition \ref{prop:maxExponentIsThetaGrowth} with $x = \theta(\omega)\in\bS^1$, gives us that $\lambda(\omega,x) = \Lambda(\nu)$. If $x\neq \theta(\omega)$, then we can use Lemma \ref{lem:lox=leta} together with \eqref{eq:160825.1} and conclude that $\lambda(\omega,x) = \lambda(\eta,\nu) = \lambda(\nu)$. 

Using Proposition \ref{prop:expintheta-expinpi} with $x = \pi(\omega)$ we see that for $\nu^{\N}$-a.e. $\omega = (f_n)_{n\in\N}\in X^{\N}_{\nu}$, 
\[
    \lim_{n\to\infty}\frac{1}{n}\log|(f_n^{-1}\circ\cdots\circ f_1^{-1})'(x)| = -\lim_{n\to\infty}\frac{1}{n}\log|(\overline{f^n_{\omega}})'(\pi(\sigma^n(\omega)))| = -\lambda(\nu).
\]
If now $x\neq \pi(\omega)$, we can apply Lemma \ref{lem:lox=leta} once more, for $\nu^-$, together with \eqref{eq:160825.1}, obtaining that
\[
    \lim_{n\to\infty}\frac{1}{n}\log|(f_n^{-1}\circ\cdots\circ f_1^{-1})'(x)| = \lambda(\eta^-,\nu^-) = -\Lambda(\nu).
\]
This finishes the proof of Theorem \ref{teo:02}. Also notice that the equalities in \eqref{eq:160825.1} is exactly the conclusion of Corollary \ref{cor:teo:02}.

\section{Exact Dimensionality}
\label{sec:exactDim}
Throughout this section we assume that $\nu$ is a probability measure on $\diff^{1+\tau}(\bS^1)$ satisfying the integrability condition \eqref{int-cond}. Additionally, we assume that $\Gamma_{\nu}$ is proximal and does not fix a point in $\bS^1$. 

Let $\eta$ be the $\nu$-stationary measure on $\bS^1$. Our goal is to prove the exactness of the dimension of $\eta$. Our approach is strongly inspired by the proof of exact dimensionality presented in \cite{HocSol:17} and in this section, we point out the main modifications in the argument needed to complete the proof of Theorem \ref{teo03}.

Fix $\varepsilon>0$ and define
    \[
        \overline{\mathcal{Z}_n}(\omega,\varepsilon) := \{
            x\in \bS^1\colon\,
            |(\overline{f^n_{\omega}})'(x)| > e^{-n(\Lambda(\nu^-) - \varepsilon)}
        \}.
    \]
\begin{proposition}
\label{prop:delta-Mn}
Let $\delta>0$.
    For $\nu^\bN$-almost every $\omega\in X_\nu^\bN$ 
    \[
    \limsup_{n\to\infty}\frac{1}{n}\sup_{
        x:
        d(x,\overline{\mathcal{Z}_n}(\omega,\varepsilon))\geq\delta
    }\log d(\overline{f^n_{\omega}}(x),\pi(\omega))\leq \lambda(\eta,\nu) + \varepsilon.
    \]
\end{proposition}
\begin{proof}
    Fix $\omega = (f_n)_{n\in \N}\in X_{\nu}^{\N}$ typical and, for every $n\geq 1$, write $g^n_{\omega} := f_n^{-1}\circ\cdots\circ f_1^{-1}$. Define,
    \[
        \mathcal{I}_k := \{
            x\in \bS^1\colon\,
            |(g_{\omega}^n)'(x)| \leq e^{n(\Lambda(\nu^-) - \varepsilon)},
            \text{ for every } n\geq k
        \}.
    \]
    Notice that, for any $k\geq 1$, $\mathcal{I}_k$ does not contain the point $\theta^-(\omega) = \pi(\omega)$. Moreover, there exists $C>0$ such that if $x\in \mathcal{I}_{k+1}\backslash \mathcal{I}_k$, then
    \[
        e^{k(\Lambda(\nu^-) - \varepsilon)}
        < |(g^k_{\omega})'(x)| \leq C e^{k(\Lambda(\nu^-) - \varepsilon)}.
    \]
    For each $k\geq 1$, let $x_k\in \mathcal{I}_k$ be a point that realizes the distance $d(\mathcal{I}_k,\theta^-(\omega))$. By Theorem \ref{teo:02}, $x_k\to \theta^-(\omega)$. Also, observe that for every $k\geq 1$, we may find $\xi_k\in \mathcal{I}_{k+1}\backslash \mathcal{I}_k$ such that
    \[
        1
        \geq d(g^k_{\omega}(x_{k+1}), g^k_{\omega}(x_k))
        = |(g^k_{\omega})'(\xi_k)|d(x_{k+1},x_k)
        \geq e^{k(\Lambda(\nu^-) - \varepsilon)}d(x_{k+1},x_k),
    \]
    which implies that $d(\theta^-(\omega), x_k) \leq e^{-k(\Lambda(\nu^-) - \varepsilon)}$ for every $k\geq 1$. In particular,
    \begin{align}
    \label{eq:031125.1}
        d(\theta^-(\omega), \mathcal{I}_k) \leq e^{-k(\Lambda(\nu^-) - \varepsilon)},
        \quad
        \text{for every}
        \quad
        k\geq 1.
    \end{align}
    
    By definition of the sets $\overline{\mathcal{Z}_k}(\omega,\varepsilon)$ and $\mathcal{I}_k$, for every $k\geq 1$ we have,
    \[
        \overline{f^k_{\omega}}(\bS^1\backslash \overline{\mathcal{Z}_k}(\omega,\varepsilon)) \subset \bS^1\backslash \mathcal{I}_k.
    \]
    Also, by Proposition \ref{prop:expintheta-expinpi}, $\pi(\sigma^k\omega) \notin \overline{\mathcal{Z}_k}(\omega,\varepsilon)$. Thus, using \eqref{eq:031125.1}, we conclude that for every $\delta>0$, 
    \[
        \frac{1}{k}\log\sup_{z\colon d(z,\overline{\mathcal{Z}_k}(\omega,\varepsilon))>\delta} d(\overline{f^k_{\omega}}(z), \pi(\omega)) \leq -\Lambda(\nu^-) + \varepsilon = \lambda(\eta,\nu) + \varepsilon.
    \]
    The result follows by taking $\limsup$ as $k\to \infty$ in both sides.

\end{proof}

\begin{proof}[Proof of Theorem \ref{teo03}]
Given an interval $J\subset \bS^1$, we write $|J|$ for its length. Fix  $\omega = (f_1, f_2, \dots) \in X_\nu^\bN$ $\nu^\bN$-typical and let $x_0 := \pi(\omega)$, where $\pi:X_\nu^\bN\to\bS^1$ is as in Theorem \ref{teo:01}. Set,
\[
    L_n:=\sup_{y\in\bS^1}\vert ((\overline{f_\omega^n})^{-1})'(y)\vert.
\]
Fix $\delta > 0$ such that $\eta(B_\delta(y))\leq 1-c$ for some $c\in(0,1)$ and for all $y \in \bS^1$ (such $\delta$ exists since $\eta$ is non-atomic). Write
\begin{align*}
    A_n^{\delta} := \left\{
        y\in \bS^1\colon\, d(y,\overline{\mathcal{Z}_n}(\omega,\varepsilon))\geq \delta
    \right\},
\end{align*}
and $\lambda:=\lambda(\nu)=\lambda(\eta,\nu)$. Using that
\[
    \lim_{n\to\infty}\frac{1}{n}\log L_n=\Lambda(\nu^-)=-\lambda.
\]
and Proposition \ref{prop:delta-Mn}, we can find a sequence $(\varepsilon_n)_{n\in\bN}$ of positive numbers converging to zero such that for every $n\geq 1$,
\[
    \exp(-\lambda n(1-\varepsilon_n)) \leq L_n \leq \exp(-\lambda n(1+\varepsilon_n)),
\]
and $\overline{f_\omega^n}(A_n^{\delta})$ is a neighborhood of $x_0$ of diameter $\exp(\lambda n(1+\varepsilon_n))$. Fix a large $N$ (which we mostly suppress in our notation) and set
\[
    r = r_N := \exp(\lambda N(1+\varepsilon_N)),
    \quad
    I_0 := B_r(x_0), 
    \quad
    \text{and}
    \quad
    I_n := (\overline{f_\omega^n})^{-1} I_0.
\]
Then $I_0$ is a neighborhood of $x_0$ and $I_n$ is a neighborhood of $x_n := (\overline{f_\omega^n})^{-1}(x_0) = \pi(\sigma^n \omega)$.

By definition of $L_n$, every interval in $\bS^1$ is expanded under $(\overline{f_\omega^n})^{-1}$ by at most $L_n$. Hence
\[
    |I_n| \leq L_n |I_0| \leq 2\exp(-\lambda n(1+\varepsilon_n)) \cdot\exp(\lambda N(1+\varepsilon_N)).
\]
Define
\[
    \hat\varepsilon_N := \frac{1}{N\lambda}\log 2 + \varepsilon_N + \sup_{1\leq n\leq N} \varepsilon_n \frac{n}{N},
\]
and notice that $\hat\varepsilon_N \to 0$ and 
\begin{equation}
\label{eq:I_N t_N}
    |I_n| \leq \exp(-\lambda n+\lambda N(1+\hat\varepsilon_N)).
\end{equation}

We do not require a lower bound for $|I_n|$, for $0\leq n <N$, but shall want one for $|I_N|$. Then, by our choice of $\varepsilon_N$, we know that $I_0 \supseteq \overline{f^N_{\omega}}(A_N^{\delta})$, whereby
\[
    I_N = (\overline{f_\omega^N})^{-1}(I_0)
    \supseteq (\overline{f_\omega^N})^{-1}\left(
        \overline{f_\omega^N} (A_N^{\delta})
    \right) = A_N^{\delta}.
\]
By definition of $\delta$, for some $c \in(0,1)$ 
\begin{equation}\label{eq:nu(IN)c}
    \eta(I_N) > c \quad \text{for all } N\in\bN. 
\end{equation}

Notice that
\[
    \nu(I_0) =\prod_{n=1}^{N}\frac{(\overline{f_\omega^{n-1}})_\ast \eta(I_0)}{(\overline{f_\omega^{n}})_\ast \eta(I_0)}\cdot(\overline{f_\omega^{N}})_\ast \eta(I_0)  =\prod_{n=1}^{N} \frac{\nu(I_{n-1})}{(f_n)_\ast \nu(I_{n-1})} \cdot \nu(I_N).
\]
Taking logarithms,
\[
    -\log \eta(I_0) = -\log \eta(I_N) + \sum_{n=1}^{N} \log \frac{(f_n)_\ast \eta(I_{n-1})}{\eta(I_{n-1})}.
\]
For $\omega = (f_j)_{j\in\N} \in X_\nu^\bN$, $\sigma^{n-1}(\omega)=(f_j)_{j\geq n}$. Hence, we see that an upper bound for
\begin{align}\label{eq:ladoderecho}
    \left\vert
        -\frac1N\log \eta(I_0)-\frac1N\sum_{n=1}^{N} \log\frac{\dd(f_n)_\ast \eta}{\dd\eta}(\pi(\sigma^{n-1}\omega))
    \right\vert
\end{align}
is given by the sum
\begin{align}
\label{eq:190325.1}
    \frac1N \log \eta(I_N)+\frac1N\sum_{n=0}^{N-1} G(\sigma^{n}\omega, \vert I_n\vert),
\end{align}
where
\begin{align*}
    G(\omega,t) := \sup\left\{
        \left\vert
            \log \frac{(f_n)_\ast \eta(J)}{\eta(J)} - \log\frac{\dd(f_n)_\ast \eta}{\dd\eta}(\pi(\omega))
        \right\vert\colon\,
        \pi(\omega)\in J, |J|\leq t
    \right\}.
\end{align*}

By \eqref{eq:nu(IN)c}, the first term in \eqref{eq:190325.1} goes to $0$ as $N\to\infty$. To deal with the second term in \eqref{eq:190325.1}, we apply Maker's theorem, with $t_{n,N}:=\exp(-\lambda n+\lambda N(1+\hat\varepsilon_N))$ (see \cite[Theorem 3.8]{HocSol:17}), to see that
\begin{align*}
    0\leq \lim_{N\to\infty}\frac1N\sum_{n=0}^{N-1} G(\sigma^{n}\omega, \vert I_n\vert)
    \leq \lim_{N\to\infty}\frac1N\sum_{n=0}^{N-1} G(\sigma^{n}\omega, \,t_{n,N} )
    = \int \lim_{t\to 0}G(\,\cdot\,, t)\, \dd\nu^{\N}.
\end{align*}
The $\nu^\bN$-integrability assumption on $\omega\mapsto\sup_tG(\omega, t)$, required to apply Maker's theorem, is ensured following the same argument as in \cite[Section 3.3]{HocSol:17} and is omitted here.

Using \cite[Proposition 3.6, item (i)]{Mal:17}, for all $\omega = (f_n)_{n\geq 1} \in X_\nu^\bN$, we have
    \[
        \lim_{t\to0^+}\sup_{I\ni \pi(\omega),\,\vert I\vert\leq t}\, \frac{(f_1)_\ast \eta(I)}{\eta(I)}=\frac{\dd(f_1)_\ast \eta}{\dd\eta}=\lim_{t\to0^+}\inf_{I\ni \pi(\omega),\,\vert I\vert\leq t}\,\frac{(f_1)_\ast \eta(I)}{\eta(I)}.
    \]
    In particular, for every $\omega\in X_{\nu}^{\N}$, $\lim_{t\to0^+}G(\omega,t)=0$. Thus, both terms in \eqref{eq:190325.1} converge to zero. Making $N\to\infty$ in \eqref{eq:ladoderecho} and applying Birkhoff's theorem we obtain
    \begin{align*}
        \lim_{N\to\infty}\frac{\log\eta(B_{r_N}(x_0))}{\log r_N}
        &= \lim_{N\to\infty}\frac{\log \eta(I_0)}{-\lambda N}
        = \lim_{N\to\infty}\frac1N\sum_{n=1}^{N} \log\frac{\dd(f_n)_\ast \eta}{\dd\eta}(\pi(\sigma^{n-1}\omega))\\
        &= \int\, \log\frac{
            \dd f_\ast \eta
        }{\dd\eta}(f(y))\, \dd(\eta\otimes\nu)(y,f)
        = h_F(\nu).
    \end{align*}
    This is enough to complete the proof.
\end{proof}

\section{examples}

Let us illustrate an instance of a proximal random dynamical system on the circle that cannot be topologically conjugate to a system induced by linear projective maps. 
\begin{example}\label{exemple:nonlinear}
Consider the support  $X_{\nu}$  of the probability measure $\nu$, and suppose that  $X_{\nu}$  contains at least two maps, say  $f$  and $g$. We discuss two situations.  

\begin{enumerate}
    \item \textbf{First case.}  
    The map  $f$  has four fixed points: two topological attractors and two topological repellers, arranged in such a way that the two repelling points are located very close to one of the attracting points.  
    On the other hand, the map $g$ is an irrational rotation.  

    \item \textbf{Second case.}  
    The map  $f$  has $N>2$ fixed points, among which at least one is a topological attractor.  
    The map $g$, in contrast, has a unique fixed point, which is topologically parabolic. This parabolic point is located very close to the attracting point of  $f$ , without coinciding with it. Moreover, it is positioned in such a way that no other attracting fixed point of  $g$  lies in between these two fixed points.  
\end{enumerate}

In both cases, the semigroup $\Gamma_{\nu}$ (generated by $X_\nu$) is proximal and has no fixed point in $\bS^1$. Moreover, if $X_\nu$ contains a map with more than two fixed points, then the induced RDS cannot be topologically conjugate to any RDS induced by projective linear maps.
\end{example}

\section*{Acknowledgments}
This research was financed by the Center of Excellence “Dynamics, Mathematical Analysis and Artificial Intelligence” at Nicolaus Copernicus University in Toruń. J.B. was also supported by the Narodowe Centrum Nauki Grant 2022/45/B/ST1/00179. We express our sincere gratitude to Katrin Gelfert for her valuable suggestions and for encouraging us throughout the elaboration process of this work.

\bibliographystyle{alpha}
\bibliography{bib}

\end{document}